\documentclass[11pt, a4paper]{amsart}

\usepackage{a4wide}
\usepackage[latin1]{inputenc} 
\usepackage[sc]{mathpazo}

\usepackage{amssymb,amsmath,amsthm}
\usepackage{mathrsfs}
\usepackage[latin1]{inputenc}
\usepackage[numbers]{natbib}
\usepackage{url}
\usepackage{stackrel}
\usepackage{txfonts}
\usepackage{hyperref} 
\hypersetup{pdftitle={},pdfauthor={Oliver Pfaffel},colorlinks=true,linkcolor=black,citecolor=black,urlcolor=black,breaklinks=true}
\usepackage[capitalise]{cleveref} 
\usepackage{hyphenat}
\usepackage{mathtools}
\usepackage{enumerate}

\renewcommand{\leq}{\leqslant}
\renewcommand{\geq}{\geqslant}

\numberwithin{equation}{section}

\newcommand{\1}{\mathbf{1}}
\newcommand{\eps}{\epsilonup}
\newcommand{\tn}[1]{\textnormal{#1}}

\newcommand{\mrm}[1]{\mathrm{#1}}

\newcommand{\N}{\mathbb{N}}
\newcommand{\Z}{\mathbb{Z}}
\newcommand{\R}{\mathbb{R}}

\newcommand{\diag}{\mathrm{diag}}
\newcommand{\T}{\mathsf{T}}

\newcommand{\sto}[2]{\stackrel[#2]{#1}{\longrightarrow}}
\newcommand{\ston}[1]{\stackrel[n\to\infty]{#1}{\longrightarrow}}

\newcommand{\norm}[1]{\left\|#1\right\|}

\newcommand{\bpm}{\begin{pmatrix}}
\newcommand{\epm}{\end{pmatrix}}

\newtheorem{theorem}{Theorem}

\newtheorem{lemma}[theorem]{Lemma}

\newtheorem{proposition}[theorem]{Proposition}
\newtheorem{remark}[theorem]{Remark}

\makeindex

\begin{document}

\title[Spectral norm of heavy-tailed random matrices]{On the spectral norm of large heavy-tailed random matrices with strongly dependent rows and columns}
\author[O. Pfaffel]{Oliver Pfaffel}
\address{TUM Institute for Advanced Study \& Department of Mathematics, Technische Universit\"at M\"unchen, Germany}
\email{o.pfaffel@gmx.de}

\date{}
\begin{abstract}
We study a new random matrix ensemble $X$ which is constructed by an application of a two dimensional linear filter to a matrix of iid random variables with infinite fourth moments. Our result gives asymptotic lower and upper bounds for the spectral norm of the (centered) sample covariance matrix $XX^\T$ when the number of columns as well es the number of rows of $X$ tend to infinity.
\end{abstract}
\maketitle

\section{Introduction and main result}


A {random matrix ensemble} is a sequence of matrices with increasing dimensions and randomly distributed entries. \emph{Random Matrix Theory} (RMT) studies the asymptotic spectrum, e.g., limiting eigenvalues and eigenvectors, of random matrix ensembles. A comprehensive introduction into RMT can be found, for instance, in the textbooks \cite{Anderson2009} and \cite{Bai2010}. 
In Davis et al. \cite{Davis2011} the authors study the asymptotic properties of the extreme singular values of a heavy-tailed random matrix $X$ the rows of which are given by independent copies of some linear process. This was motivated by the statistical analysis of observations of a high-dimensional linear process with independent components. Typically, the linear processes used in multivariate stochastic modeling have the more general form
\begin{align*}
 \mathbf{X}_t = \sum_j A^{(j)} \mathbf{Z}_{t-j}, \quad t=1,\ldots,n,
\end{align*}
where $A^{(j)}$ is a sequence of deterministic $p\times p$ matrices and $\mathbf{Z_t}$ is a noise vector containing $p$ independent and identically distributed (iid) random variables $Z_{1t},\ldots,Z_{pt}$.
Of course, the process $\bf X$ does not have independent components except when $A^{(j)}$ is a multiple of the identity matrix. Let us denote by $\tilde X$ the matrix with columns $\bf X_1,\ldots,\bf X_n$. Then the $it$-th entry of $\tilde X$ is given by
\begin{align*}
 \tilde X_{it} = \sum_j \sum_{k=1}^p A_{ik}^{(j)} Z_{k,t-j}.
\end{align*}
This motivates to study the 
general random matrix ensemble
\begin{align*}
 \tilde X_{it} = \sum_j \sum_{k} d(i,j,k)  Z_{i-k,t-j}
\end{align*}
with some iid array $Z=(Z_{it})$ and some function $d:\N\times\Z^2\to\R,(i,j,k)\mapsto d(i,j,k)$ such that the above double sum converges. The matrix $\tilde X$ can be seen as a two dimensional filter applied to some noise matrix $Z$. The spectral distribution of these matrices has been studied for Gaussian matrices $\tilde X$ and $ d(i,j,k)=\tilde d(j,k)$ by \cite{hachem2005}, and for more general light-tailed distributions by \cite{anderson2008} under the assumption that $\tilde d(j,k)=0$ if $j$ or $k$ is larger than some fixed constant. We investigate the case where the function $d$ can be factorized in the form $d(i,j,k)=c_j \theta_k$. Thus in our model the random matrix $\hat X = (\hat X_{it})\in\R^{p\times n}$ is given by
\begin{align}\label{mymodel}
\hat X_{it} = \sum_j \sum_{k} c_j \theta_k  Z_{i-k,t-j},
\end{align}
 for two real sequences $(c_j)$ and $(\theta_k)$. In contrast to the model $X=(X_{it})$ considered in Davis et al. \cite{Davis2011}, with 
 \[ X_{it}=\sum_j c_j Z_{i,t-j}, \]
the matrix $\hat X$ has not only dependent columns but also dependent rows. Indeed, writing the model \eqref{mymodel} in the form
\begin{align}
 \hat X_{it} &= \sum_j c_j \xi_{i,t-j}, \label{m1}\\
 \xi_{it} &= \sum_{k} \theta_k Z_{i-k,t}, \label{m2}
\end{align}
one can see that, by going from $X$ to $\hat X$, the noise sequence $Z$ in the processes along the rows is replaced by a linear process $\xi$ along the columns. Since we want to investigate a heavy-tailed random matrix model we assume that $(Z_{it})_{i,t}$ 
 is an array of regularly varying iid random variables with tail index $\alpha\in(0,4)$ satisfying
\begin{align}
 nP(|Z_{11}|>a_n x) \to x^{-\alpha}. \label{mZ}
\end{align}
Furthermore, let $(c_j)$ and $(\theta_k)$ be sequences of real numbers such that
\begin{align}
 &\sum_j |c_j|^\delta<\infty, \tn{ and}\label{mC} \\
 &\sum_{k} |\theta_k|^\delta < \infty \quad\tn{for some }\delta<\min\{\alpha,1\}. \label{mTheta}
\end{align}
If $5/3<\alpha<4$ we also require that $Z_{11}$ satisfies the tail balancing condition, i.e., the existence of the limits
\begin{align}
\lim_{x\to\infty}\frac{P(Z_{11}>x)}{P(|Z_{11}|>x)}=q \quad\tn{and}\quad \lim_{x\to\infty}\frac{P(Z_{11}\leq -x)}{P(|Z_{11}|>x)}=1-q
\label{balancing3}
\end{align}
for some $0\leq q\leq 1$. By the above definitions, $\hat X$ is a $p\times n$ random matrix with dependent entries with infinite fourth moments. Under the assumption that $p$ and $n$ go to infinity such that the ratio $p/n$ converges to a positive finite constant, Soshnikov \cite{Soshnikov2004,Soshnikov2006} and Auffinger et al. \cite{Auffinger2009} have studied the eigenvalues of heavy-tailed random matrices with independent and identically distributed entries. Bose et al.\cite{bose2009} investigate the spectral norm of circulant type matrices with heavy-tailed entries.
\noindent
In the following we assume that both $p=p_n$ and $n$ go to infinity such that
\begin{align}
\limsup_{n\to\infty} \frac{p_n}{n^\beta}<\infty
\label{mBeta}
\end{align}
for some $\beta>0$ satisfying
\begin{align*}
&\beta<\infty \quad\tn{if}\quad \alpha\in(0,1], \\
&\beta<\max\left\{\frac{2-\alpha}{\alpha-1},\frac12\right\} \quad\tn{if}\quad \alpha\in(1,2), \\
&\beta<\max\left\{\frac{4-\alpha}{4(\alpha-1)},\frac{1}{3}\right\}\quad\tn{if}\quad 2\leq\alpha< 3, \quad\tn{or} \\
&\beta<\frac{4-\alpha}{3\alpha-4} \quad\tn{if}\quad 3\leq\alpha<4.
\end{align*}
\noindent
Recall that any symmetric matrix $A$ has real eigenvalues. The spectral norm $\norm{A}_2$ of $A$ is given by the maximum of the absolute values of the eigenvalues of $A$. For $\hat X$ given by \eqref{mymodel}, our main theorem investigates the asymptotic behaviour of the spectral norm $\norm{S}_2$ of the \emph{centered sample covariance matrix} $S=\hat X \hat X^\T-n\mu_{X,\alpha} HH^\T$, where
\begin{align}\label{mu hat X alpha}
\mu_{X,\alpha} = \left\{\begin{array}{ll}  0 & \tn{ for } 0<\alpha<2,\\
 E\left(Z_{11}^2\1_{\{Z_{11}^2\leq a_{np}^2\}}\right) \sum_j c_j^2 &  \tn{ for } \alpha=2 \tn{ and } EZ_{11}^2=\infty,\\
 E\left(Z_{11}^2\right) \sum_j c_j^2 & \tn{ else,} \end{array}\right.
\end{align}
and $H=(H_{ij})\in\R^{p\times 3p}$ is given by
\begin{align}\label{H}
H_{ij}=\theta_{p-(j-i)}\1_{\{ 0\leq j-i\leq 2p \}}.
\end{align}
Observe that the diagonal entries of $n\mu_{X,\alpha}HH^\T$ are exactly the means of the diagonal elements of $\hat X\hat X^\T$ if the observations have a finite variance. In case the observations have an infinite variance, we do not have to center, except when $\alpha=2$ and $EZ_{11}^2=\infty$, where we use a truncated version of the mean. In the latter case $\mu_{X,\alpha}$ also depends on $p$ and $n$.

\begin{theorem}\label{theorem dependent rows}
Consider the random matrix model given by equations \eqref{mymodel}, \eqref{mZ}, \eqref{mC} and \eqref{mTheta} with $\alpha\in(0,4)$. If $\alpha\in(5/3,4)$ we assume that $Z_{11}$ has zero mean and satisfies the tail balancing condition \eqref{balancing3}. 
Denote by $S=\hat X \hat X^\T-n\mu_{X,\alpha} HH^\T$ the centered sample covariance matrix, with $\mu_{X,\alpha}$ and $H=(H_{ij})\in\R^{p\times 3p}$ as given in \eqref{mu hat X alpha} and \eqref{H}.
Let $\Gamma_1$ be an exponentially distributed random variable with mean one and $x>0$.
If $p$ and $n$ go to infinity such that condition \eqref{mBeta} is satisfied then we have for the spectral norm $\norm{S}_2$ of $S$ that
 \begin{align}\label{bounds for lmax}
 P\left( \Gamma_1^{-2/\alpha} \max_{k}\theta_k^2  \sum_j c_j^2 >x \right) \leq& \liminf_{n\to\infty} P\left( \norm{S}_2 > a_{np}^2 x \right) \nonumber\\
\leq& \limsup_{n\to\infty} P\left( \norm{S}_2 > a_{np}^2 x \right) \nonumber\\ \leq& P\left( \Gamma_1^{-2/\alpha} \max_{l} |\theta_{l}| \sum_{k} |\theta_{k}| \sum_j c_j^2>x \right)
\end{align}
\end{theorem}

\begin{remark}
\begin{enumerate}
\item If all $\theta_k$'s except one are zero, one has equality and therefore recovers the result from \cite[Theorem 1]{Davis2011}. If two or more $\theta_k$ are non-zero, then 
\[ P\left( \Gamma_1^{-2/\alpha} \max_{k}\theta_k^2  \sum_j c_j^2 >x \right) < P\left( \Gamma_1^{-2/\alpha} \max_{l} |\theta_{l}| \sum_{k} |\theta_{k}| \sum_j c_j^2>x \right). \]
Whether the $\liminf$ and $\limsup$ are equal in this case and attain one of its boundaries remain open problems.
\item Since $P(\Gamma_1^{-2/\alpha}\leq x)=e^{-x^{-\alpha/2}}$, inequality \eqref{bounds for lmax} can equivalently be written as
 \begin{align*}
 \exp\left(-x^{-\alpha/2}\max_{l} |\theta_{l}|^{\alpha/2}   \left(\sum_{k} |\theta_{k}| \sum_j c_j^2\right)^{\alpha/2} \right) \leq& \liminf_{n\to\infty} P\left( \norm{S}_2 \leq a_{np}^2 x \right) \\
\leq& \limsup_{n\to\infty} P\left( \norm{S}_2 \leq a_{np}^2 x \right) \\
\leq& \exp\left(-x^{-\alpha/2} \max_{k}|\theta_k|^\alpha  \left(\sum_j c_j^2\right)^{\alpha/2} \right).
\end{align*}
\end{enumerate}
\end{remark}

Results from the theory of point processes and regular variation are required through most of this paper.
A detailed account on both topics can be found in a number of texts. We mainly adopt the setting, including notation and terminology, of Resnick \cite{Resnick2008}.

\section{Dependence of successive rows}

To understand the basic principle of our method it is beneficial to first investigate the case where only successive rows of $\hat X$ are dependent and where $\alpha\in(0,2)$.
Since $\mu_{X,\alpha}=0$ for $\alpha<2$, $S=XX^\T$ and therefore the spectral norm of $S$ is equal to the largest eigenvalue of $XX^\T$, i.e., $\norm{S}_2=\lambda_{\max}$. We start with the model
\begin{align}
 \hat X_{it} &= \sum_j c_j \xi_{i,t-j}, \\
\xi_{it} &= Z_{it} + \theta Z_{i-1,t}.
\end{align}
It is easy to see that $\hat X_{it} = X_{it} + \theta X_{i-1,t}$, where $X_{it}=\sum_j c_j Z_{i,t-j}$ for $i=0,1,\ldots,p$, and $t=1,\ldots,n$. To proceed further we define the matrices $\hat X=(\hat X_{it})\in\R^{p\times n}$, $X=(X_{(i-1),t})\in\R^{(p+1)\times n}$ and $H=(H_{ij})\in\R^{p\times (p+1)}$, where all entries of $H$ are zero except $H_{ii}=\theta$ and $H_{i,i+1}=1$. Then we clearly have the matrix equality
\begin{align}\label{matrix equality}
 \hat X = H X.
\end{align}
Moreover, we denote by $D=(D_{i})=\diag(XX^\T)\in\R^{(p+1)\times(p+1)}$ the diagonal of $XX^\T$, that is the diagonal matrix which consists of the diagonal entries of $XX^\T$. For the convenience of the reader, we restate the result from \cite[Proposition 3.4]{Davis2011}.
\begin{proposition}\label{rephrased op norm conv}
Under the conditions of \cref{theorem dependent rows} we have that
\begin{align*}
  a_{np}^{-2}\norm{XX^\T-D}_2 \ston{P} 0.
\end{align*}
\end{proposition}
Thus, since $\norm{H}_2\leq\norm{H}_\infty\leq 1+|\theta|$, we immediately conclude, by \eqref{matrix equality}, that
\begin{align}\label{opnorm2}
 a_{np}^{-2}\norm{\hat X \hat X^\T-HDH^\T}_2\leq \norm{H}_2^2 a_{np}^{-2}\norm{XX^\T-D}_2 \to 0. 
\end{align}
Hence, by Weyl's inequality (\cite[Corollary III.2.6]{Bhatia1997}), the largest eigenvalue $\lambda_{\max}$ of the sample covariance matrix $\hat X \hat X^\T$ based on the observations $\hat X$ is asymptotically equal to the largest eigenvalue of the tridiagonal matrix
\begin{align}
 HDH^\T=\begin{pmatrix} D_{1}+\theta^2 D_{2} & \theta D_{2} & 0 &   \\ 
			\theta D_{2} & D_{2}+\theta^2 D_{3} & \theta D_{3} &   \\
			0 & \ddots & \ddots  &  0  \\ 
			 & & D_{p-1}+\theta^2 D_{p}&  \theta D_{p}  \\
			 & 0 & \theta D_{p} & D_{p}+\theta^2 D_{p+1}
\end{pmatrix}\in\R^{p\times p}.
\end{align}
It is our goal to find an asymptotic upper and lower bound for $\lambda_{\max}$. First we prove a lower bound. Clearly, $\lambda_{\max}$ is asymptotically larger or equal than the largest diagonal entry of $HDH^\T$, i.e.,
\begin{align}\label{lower bound ma1}
 \lambda_{\max}\geq \max_{1\leq i\leq p} (D_{i}+\theta^2 D_{i+1}) + o_P(a_{np}^2),
\end{align}
where $o_P(1)$ denotes some generic random variable that converges to zero in probability as $n$ goes to infinity. Since $D_{i+1}=\sum_{t=1}^n X_{it}^2$, we have to find the maximum of an MA(1) process of partial sums of linear processes. By \cite[Proposition 3.5]{Davis2011} we already know that
\begin{align}\label{known pp conv}
 \sum_{i=0}^p\eps_{a_{np}^{-2}D_{i+1}} = \sum_{i=0}^p\eps_{a_{np}^{-2}\sum_{t=1}^n X_{it}^2} \ston{D} \sum_{i=1}^\infty \eps_{\Gamma_i^{-2/\alpha}\sum_j c_j^2}.
\end{align}
Since $(D_i)$ is an iid sequence, this result can be generalized as follows.

\begin{lemma}\label{ma of partial sum}
Under the conditions of \cref{theorem dependent rows} we have that
\[ I_p=\sum_{i=1}^{p}\eps_{a_{np}^{-2}(D_{i+1},D_i)} \ston{D} I= \sum_{i=1}^\infty \left(\eps_{\Gamma_i^{-2/\alpha}\sum_j c_j^2(1,0)}+ \eps_{\Gamma_i^{-2/\alpha}\sum_j c_j^2(0,1)}\right). \]
\end{lemma}
\begin{proof}
By the continuous mapping theorem applied to \eqref{known pp conv}, we immediately conclude that
\begin{align*}
 I_p^* = \sum_{i=1}^{p}\left(\eps_{a_{np}^{-2}(D_{i+1},0)} + \eps_{a_{np}^{-2}(0,D_i)} \right)\ston{D} I.
\end{align*}
Thus, we only have to show that $|I_p(f)-I_p^*(f)|\to 0$ in probability for any continuous function with $\mathrm{supp}(f)\subset\{x=(x_1,x_2)\in\R^2:\max\{|x_1|,|x_2|\}\geq\delta\}$. To this end, let $L=\{x: \min\{|x_1|,|x_2|\}<\delta\}$ and observe that, by independence of $(D_i)$,
\begin{align*}
 EI_p(L^c)\leq pP(|D_{i+1}|\geq a_{np}^2\delta,|D_{i}|\geq a_{np}^2\delta)= O(\delta^{-\alpha}p^{-1})\to 0.
\end{align*}
Thus $I_p(f)=\int_L f dI_p+ o_P(1)$ and, by definition of $I_p^*$, $I_p^*(f)=\int_L f dI_p^*$. Since $f(z)=0$ if $\max\{|x_1|,|x_2|\}<\delta$, it suffices to show that
\begin{align*}
A + B =& \sum_{i=1}^{p} \left| f(a_{np}^{-2}(D_{i+1},D_i)) \1_{\{a_{np}^{-2}|D_{i+1}|\geq\delta\}\cap\{a_{np}^{-2}|D_{i}|<\delta\}} - f(a_{np}^{-2}(D_{i+1},0)) \1_{\{a_{np}^{-2}|D_{i+1}|\geq\delta\}} \right| \\
&+ \sum_{i=1}^{p} \left| f(a_{np}^{-2}(D_{i+1},D_i)) \1_{\{a_{np}^{-2}|D_{i+1}|<\delta\}\cap\{a_{np}^{-2}|D_{i}|\geq\delta\}} - f(a_{np}^{-2}(0,D_i)) \1_{\{a_{np}^{-2}|D_{i}|\geq\delta\}} \right| \ston{P} 0.
\end{align*}
We only treat term $A$, as $B$ can be handled essentially the same way. To this end, observe that
\begin{align*}
 A\leq&\sum_{i=1}^{p} \left| f(a_{np}^{-2}(D_{i+1},D_i)) - f(a_{np}^{-2}(D_{i+1},0)) \right| \1_{\{a_{np}^{-2}|D_{i+1}|\geq\delta\}\cap\{a_{np}^{-2}|D_{i}|<\delta\}} \\
&+ \sum_{i=1}^{p} |f(a_{np}^{-2}(D_{i+1},0))| \1_{\{a_{np}^{-2}|D_{i+1}|\geq\delta\}\cap\{a_{np}^{-2}|D_{i}|\geq\delta\}} = I+II.
\end{align*}
Clearly, by independence,
\begin{align*}
 E(II)\leq \sup f(x) p P(a_{np}^{-2}|D_{i+1}|\geq\delta)P(a_{np}^{-2}|D_{i}|\geq\delta)=O(p^{-1}) \to 0.
\end{align*}
Furthermore, we have, for any $0<\eta<\delta$, that 
\[  \1_{\{a_{np}^{-2}|D_{i+1}|\geq\delta\}\cap\{a_{np}^{-2}|D_{i}|<\delta\}} \leq  \1_{\{a_{np}^{-2}|D_{i+1}|\geq\delta\}\cap\{a_{np}^{-2}|D_{i}|<\eta\}} +  \1_{\{a_{np}^{-2}|D_{i+1}|\geq\eta\}\cap\{a_{np}^{-2}|D_{i}|\geq\eta\}}. \]
Thus, for some $c>0$,
\begin{align*}
 E(I)\leq& \sup\{|f(x_1,x_2)-f(x_1,0)|: |x_1|>\delta,|x_2|<\eta\}pP(|D_{i+1}|\geq a_{np}^2\eta) \\ &+ cp P(|D_{i+1}|\geq a_{np}^2\eta)P(|D_{i}|\geq a_{np}^2\eta).
\end{align*}
Obviously, the second summand converges, for fixed $\eta>0$, to zero as $n\to\infty$. The first summand can be made arbitrarily small by choosing $\eta$ small enough, since $f$ is uniformly continuous.
\end{proof}
The continuous mapping theorem applied to \cref{ma of partial sum} gives
\[ \sum_{i=1}^p\eps_{a_{np}^{-2}(\theta^2 D_{(i+1)}+D_i)} \ston{D} \sum_{i=1}^\infty \left(\eps_{\Gamma_i^{-2/\alpha}\sum_j c_j^2\theta^2}+\eps_{\Gamma_i^{-2/\alpha}\sum_j c_j^2}\right). \]
Therefore, by \eqref{lower bound ma1}, the asymptotic lower bound of $\lambda_{\max}$ is given by
\begin{align}\label{ma1 lower}
a_{np}^{-2} \max_{1\leq i\leq p} (D_{i}+\theta^2 D_{i+1}) \ston{D} \max\{1,\theta^2\} \Gamma_1^{-2/\alpha} \sum_j c_j^2.
\end{align}
Regarding the upper bound, we make use of the fact that $\norm{HDH^\T}_2\leq\norm{HDH^\T}_{\infty}$. Observe that
\begin{align*}
\norm{HDH^\T}_{\infty} =&  \max_{1\leq i\leq p} \left( \1_{\{i\neq 1\}} |\theta| D_i +  D_{i}+\theta^2 D_{i+1} + |\theta| D_{i+1} \1_{\{i\neq p\}}  \right) \\
=& \max_{1\leq i\leq p} \left( (1+|\theta|\1_{\{i\neq 1\}} ) D_i + (|\theta|\1_{\{i\neq p\}} +\theta^2) D_{i+1} \right). 
\end{align*}
So once again we have to determine the maximum of an MA(1) of partial sums of linear processes. An application of \cref{ma of partial sum} yields that
\begin{align}\label{ma1 upper}
a_{np}^{-2} \norm{HDH^\T}_{\infty} \ston{D} \max\{1+|\theta|,|\theta|+\theta^2\} \Gamma_1^{-2/\alpha} \sum_j c_j^2.
\end{align}
The lower and upper bound \eqref{ma1 lower} and \eqref{ma1 upper} together with equation \eqref{opnorm2} finally yield that
\begin{align*} 
 P\left( \max\{1,\theta^2\} \Gamma_1^{-2/\alpha} \sum_j c_j^2 >x \right) \leq& \liminf_{n\to\infty} P\left( \lambda_{\max} > a_{np}^2 x \right) \\
\leq& \limsup_{n\to\infty} P\left( \lambda_{\max} > a_{np}^2 x \right) \\ \leq&  P\left( \left(|\theta|+\max\{1,\theta^2\}\right) \Gamma_1^{-2/\alpha} \sum_j c_j^2>x \right).
\end{align*}
Clearly, this result is a special case of \cref{theorem dependent rows} when the process $\xi_{it}$ is a moving average process of order one.

\section{Proof of the theorem}

In this section we will proof \cref{theorem dependent rows} in its full generality. We start with the case where $\alpha<2$. To this end we define an approximation $\hat X^{(p)}$ of $X$ and 
so that 
\begin{align}
\tn{(i)}&\quad a_{np}^{-2}\norm{\hat X^{(p)} (\hat X^{(p)})^\T - HDH^\T}_2\ston{P}0, \label{opnormconv1} \\
\tn{(ii)}&\quad a_{np}^{-2}\norm{\hat X \hat X^\T - \hat X^{(p)} (\hat X^{(p)})^\T}_2\ston{P}0, \label{opnormconv2} \\
\tn{(iii)}&\quad \tn{and finally we derive upper and lower bounds for } \norm{HDH^\T}_2. \nonumber
\end{align}
Note that, for notational convenience, we will assume that $\theta_k=0$ for $k<0$, since the extension of the proof to the case where the dependence in \eqref{m2} is two-sided is analogous.\\

\noindent (i).
First we define the approximation $\hat X^{(p)} = (\hat X_{it}^{(p)})\in\R^{p\times n}$ by $\hat X_{it}^{(p)}=\sum_{k=0}^p \theta_k X_{i-k,t}$, where $X_{it}=\sum_j c_j Z_{i,t-j}$. Furthermore we define $X = (X_{i-p,t})\in\R^{2p\times n}$, and $H=(H_{ij})\in\R^{p\times 2p}$ by
\begin{align}
 H_{ij} = \left\{\begin{array}{cl} \theta_{p-(j-i)} & \tn{ if } 0\leq j-i\leq p, \\
				    0		    & \tn{ else}.
                 \end{array}\right.
\end{align}
Then we have that $HX=\hat X^{(p)}$. Indeed,
\begin{align*}
 (HX)_{it} =& \sum_{l=0}^{2p} H_{il} X_{l-p,t} = \sum_{l=i}^{i+p} H_{il} X_{l-p,t} = \sum_{l=0}^p H_{i,i+l} X_{i+l-p,t} = \sum_{l=0}^p \theta_{p-l} X_{i-(p-l),t} 
\\ =& \sum_{k=0}^p \theta_k X_{i-k,t} = \hat X_{it}^{(p)}.
\end{align*}
Thus, if we let $D=(D_{i})=\diag(XX^\T)\in\R^{2p\times 2p}$, then we obtain \eqref{opnormconv1} by virtue of \cref{rephrased op norm conv} and $\norm{H}_2\leq\norm{H}_\infty\leq\sum_{k=0}^\infty|\theta_k|<\infty$ .\\

\noindent (ii).
In order to proceed we will require the following lemma.
\begin{lemma}\label{ma of partial sum 2}
Under the conditions of \cref{theorem dependent rows} we have, for $0<\alpha<2$, that
 \begin{align*}
\sum_{i=1}^p \eps_{a_{np}^{-2}\sum_{k=0}^\infty \theta_{k}\sum_{t=1}^n X_{i-k,t}^2} \ston{D} \sum_{i=1}^\infty \sum_{k=0}^\infty \eps_{\Gamma_i^{-2/\alpha} \theta_k \sum_j c_j^2 }.
\end{align*}
\end{lemma}
\begin{proof}
%
A straight-forward generalization of \cref{ma of partial sum} yields, for any $m<\infty$, that
\begin{align}
  \sum_{i=1}^{p}\eps_{a_{np}^{-2}\sum_{t=1}^n(X_{it}^2,X_{i-1,t}^2,\ldots,X_{i-m,t}^2)} \ston{D} \sum_{k=0}^{m} \sum_{i=1}^\infty \eps_{\Gamma_i^{-2/\alpha}\sum_j c_j^2 e_{k+1}},
\end{align}
where $e_k$ denotes the $k$-th unit vector of $\R^\infty$, i.e, the $k$-th component of $e_k$ is one and all others are zero.
By an application of the continuous mapping theorem we obtain the claim for a finite order moving average of the partial sums $(\sum_{t=1}^n X_{it}^2)_i$, i.e.,
 \begin{align*}
\sum_{i=1}^p \eps_{a_{np}^{-2}\sum_{k=0}^m \theta_{k} \sum_{t=1}^n X_{i-k,t}^2} \ston{D} \sum_{i=1}^\infty \sum_{k=0}^{m} \eps_{\Gamma_i^{-2/\alpha} \theta_{k} \sum_j c_j^2},
\end{align*}
On the other hand we have, for $m\to\infty$, that
\begin{align*}
\sum_{i=1}^\infty \sum_{k=0}^{m} \eps_{\Gamma_i^{-2/\alpha} \theta_{k} \sum_j c_j^2} \sto{D}{m\to\infty} \sum_{i=1}^\infty \sum_{k=0}^{\infty} \eps_{\Gamma_i^{-2/\alpha} \theta_{k} \sum_j c_j^2}.
\end{align*}
To finish the proof of the lemma it is, by \cite[Theorem 3.2]{Billingsley1999}, therefore only left so show that
\begin{align*}
 \lim_{m\to\infty}\limsup_{n\to\infty} \rho\left( \sum_{i=1}^p \eps_{a_{np}^{-2}\sum_{k=0}^m \theta_{k} \sum_{t=1}^n X_{i-k,t}^2}, \sum_{i=1}^p \eps_{a_{np}^{-2}\sum_{k=0}^\infty \theta_{k}\sum_{t=1}^n X_{i-k,t}^2}  \right) = 0,
\end{align*}
where $\rho$ denotes a metric of the vague topology on the space of point processes. To this end, observe that
\begin{align*}
 \left|\sum_{k=0}^m \theta_{k}\sum_{t=1}^n X_{i-k,t}^2 - \sum_{k=0}^\infty \theta_{k}\sum_{t=1}^n X_{i-k,t}^2 \right| 
 \leq \sum_{k>m}|\theta_{k}|\sum_{t=1}^n X_{i-k,t}^2.
\end{align*}
Therefore, by the arguments of the proof of \cite[Proposition 3.5]{Davis2011}, we only have to show, for any $\gamma>0$, that 
\[ \lim_{m\to\infty}\limsup_{n\to\infty} P((A_n^\gamma)^c)=0, \]
where
\begin{align*}
 A_n^\gamma = \left\{ \max_{1\leq i\leq p} \sum_{l>m}|\theta_{l}|\sum_{t=1}^n X_{i-l,t}^2 \leq a_{np}^2\gamma \right\}.
\end{align*}
Observe that
\begin{align}
P((A_n^\gamma)^c) \leq& pP\left(\sum_{l>m}|\theta_{l}|\sum_{t=1}^n X_{lt}^2 > a_{np}^2\gamma\right) \leq pP\left(\sum_{l>m}|\theta_{l}| \sum_j c_j^2 \sum_{t=1}^n Z_{l,t-j}^2 > a_{np}^2 \frac{\gamma}{2}\right) \nonumber\\
&+ pP\left(\sum_{l>m}|\theta_{l}| \sum_j \sum_{k>j} |c_j c_k| \sum_{t=1}^n |Z_{l,t-j}Z_{l,t-k}| > a_{np}^2 {\gamma}\right) = \mrm{I} + \mrm{II}. \label{twoterms}
\end{align}
We have
\begin{align*}
 \lim_{m\to\infty}\limsup_{n\to\infty} I = \lim_{m\to\infty} \left( \sum_{l>m}|\theta_{l}| \right)^{\alpha/2} \left ( 2 \sum_j c_j^2 \right)^{\alpha/2} \gamma^{-\alpha/2}=0
\end{align*}
by a slight modification of the proof of \cite[Lemma 3.1]{Davis2011}. In fact, one can also map the array $(Z_{it})$ to a sequence and then apply \cite[Lemma 3.1]{Davis2011} directly. Regarding the second term, note that
\begin{align*}
 II \leq& pP\left(\sum_{l>m}|\theta_{l}| \sum_j \sum_{k>j} |c_j c_k| \sum_{t=1}^n Z_{l,t-j}^2 > a_{np}^2 {\gamma}\right) \\
+& pP\left(\sum_{l>m}|\theta_{l}| \sum_j \sum_{k>j} |c_j c_k| \sum_{t=1}^n Z_{l,t-k}^2 > a_{np}^2 {\gamma}\right) = \mrm{II}_1 + \mrm{II}_2.
\end{align*}
As before we conclude that
\begin{align*}
  \lim_{m\to\infty}\limsup_{n\to\infty} \mrm{II}_1 = \lim_{m\to\infty} \left( \sum_{l>m}|\theta_{l}| \right)^{\alpha/2} \left ( \sum_j \sum_{k>j} |c_j c_k| \right)^{\alpha/2} \gamma^{-\alpha/2}=0,
\end{align*}
and clearly term $\mrm{II}_2$ can be handled similarly.
\end{proof}
We will now prove equation \eqref{opnormconv2}. By definition of the matrices $\hat X$ and $\hat X^{(p)}$ we have that
\[ (\hat X \hat X^\T - \hat X^{(p)} (\hat X^{(p)})^\T)_{ij}= \sum_{l,l'k,k'\in\Z^2\times(\N_0\backslash\{0,1,\ldots,p\})^2} c_l c_{l'} \theta_k \theta_{k'} \sum_{t=1}^n Z_{i-k,t-l} Z_{j-k',t-l'}. \]
Therefore we have the bound
\begin{align*}
\norm{\hat X \hat X^\T - \hat X^{(p)} (\hat X^{(p)})^\T}_2 \leq& \norm{\hat X \hat X^\T - \hat X^{(p)} (\hat X^{(p)})^\T}_\infty \\ =&
 \max_{1\leq i\leq p} \sum_{j=1}^p \sum_{l,l',k,k'\in\Z^2\times(\N_0\backslash\{0,1,\ldots,p\})^2} |c_l c_{l'} \theta_k \theta_{k'}| \sum_{t=1}^n |Z_{i-k,t-l} Z_{j-k',t-l'}|.
 \end{align*}
Observe that the product $|Z_{i-k,t-l} Z_{j-k',t-l'}|$ has tail index $\alpha/2$ if and only if $j-k'=i-k$ and $l=l'$. In this case we can treat this term like the first term in $\mrm{I}$ in \eqref{twoterms} and obtain
\[ a_{np}^{-2} \max_{1\leq i\leq p} \sum_{l,k,k'\in\Z\times\{p+1,p+2,\ldots\}^2} |c_l^2 \theta_k\theta_{k'}| \sum_{t=1}^n |Z_{i-k,t-l}^2| \ston{P} 0, \]
since $\sum_{k>p}|\theta_k|\to 0$. If the product $|Z_{i-k,t-l} Z_{j-k',t-l'}|$ does not have tail index $\alpha/2$, i.e.,  $j-k'\neq i-k'$ or  $l\neq l'$, then the product has only tail index $\alpha$ and can then be treated similarly as the second term $\mrm{II}$ in \eqref{twoterms}.\\

\noindent (iii).
By a combination of (i) and (ii) we have that
\[ a_{np}^{-2}\norm{\hat X \hat X^\T - HDH^\T}_2\ston{P}0. \]
Thus, by Weyl's inequality, the difference of the largest eigenvalues of $\hat X \hat X^\T$ and $HDH^\T$ converges to zero. As in the previous section, the final step is to find lower and upper bounds on $\norm{HDH^\T}_2$. By definition of $H$, we have
\begin{align*}
 (HDH^\T)_{ij}=\sum_{l=\max\{i,j\}}^{\min\{i,j\}+p}{\theta_{p-(l-i)}\theta_{p-(l-j)}D_l}.
\end{align*}
Hence $HDH^\T$ is no longer a tridiagonal matrix. Recall that the entries of the diagonal matrix $D$ are given by $D_{i}=\sum_{t=1}^n X_{i-p,t}^2$. By virtue of \cref{ma of partial sum 2} an asymptotic lower bound is given by
\begin{align*}
 a_{np}^{-2} \norm{HDH^\T}_2\geq& a_{np}^{-2} \max_{1\leq i\leq p} (HDH^\T)_{ii} \\
=& a_{np}^{-2} \max_{1\leq i\leq p} (\theta_p^2 D_i +\ldots +\theta_0^2 D_{i+p}^2) \ston{D} \Gamma_1^{-2/\alpha} \max_{k}\theta_k^2 \sum c_j^2.
\end{align*}
Regarding the upper bound, observe that
\begin{align*}
  \norm{HDH^\T}_2\leq& \norm{HDH^\T}_\infty=\max_{1\leq i\leq p}\sum_{j=1}^p|(HDH^\T)_{ij}| \\
\leq& \max_{1\leq i\leq p}\sum_{j=1}^p \sum_{l=\max\{i,j\}}^{l=\min\{i,j\}+p}{|\theta_{p-(l-i)}\theta_{p-(l-j)}|D_l} \\
=& \max_{1\leq i\leq p}\sum_{l=1}^{2p} D_l \sum_{j=1}^p \1_{\{ l-p\leq j\leq l, i\leq l\leq i+p \}} |\theta_{p-(l-i)}\theta_{p-(l-j)}| \\
=& \max_{1\leq i\leq p}\sum_{l=i}^{i+p} D_l |\theta_{p-(l-i)}| \sum_{j=l-p}^l |\theta_{p-(l-j)}| \\
=& \max_{1\leq i\leq p}\sum_{l=0}^{p} D_{i+l} |\theta_{p-l}| \sum_{k=0}^p |\theta_{k}|,
\end{align*}
so we have to determine the maximum of a moving average of order $p$ of $(D_i)$, with coefficients $|\theta_{p-l}| \sum_{k=0}^p |\theta_{k}|$. By \cref{ma of partial sum 2},
\begin{align}
 a_{np}^{-2} \max_{1\leq i\leq p}\sum_{l=0}^{p} D_{i+l} |\theta_{p-l}| \sum_{k=0}^p |\theta_{k}| \ston{D} \Gamma_1^{-2/\alpha} \max_{0\leq l\leq \infty} |\theta_{l}| \sum_{k=0}^\infty |\theta_{k}| \sum_j c_j^2.
\end{align}
This completes the proof of \cref{theorem dependent rows} for $\alpha<2$.

\begin{proof}[Proof of \cref{theorem dependent rows} for $\alpha\geq 2$]
Since we now consider the spectral norm 
of $\hat X \hat X^\T-n\mu_{X,\alpha} HH^\T$, one has to replace $D$ by the centered diagonal matrix $\tilde D=D-n\mu_{X,\alpha} I_p$, i.e, 
\[ \tilde D_i=\sum_{t=1}^n(X_{i-p,t}^2-\mu_{X,\alpha}).\]
Then one has, with the same truncation as before, that
\begin{align*}
a_{np}^{-2} \norm{(\hat X^{(p)} (\hat X^{(p)})^\T-n\mu_{X,\alpha} HH^\T)-H\tilde D H^\T}_2=& a_{np}^{-2} \norm{H(X X^\T-n\mu_{X,\alpha} I_p)H^\T-H (D-n\mu_{X,\alpha} I_p) H^\T}_2 \\
\leq& 
\norm{H}_2^2 a_{np}^{-2} \norm{X X^\T - D }_2 
\ston{P} 0,
\end{align*}
by an application of \cref{rephrased op norm conv}. Then one shows, similarly as in \cref{ma of partial sum}, that for each $m<\infty$,
\begin{align*}
\sum_{i=1}^p \eps_{a_{np}^{-2}|\sum_{k=0}^m \theta_{k}\sum_{t=1}^n (X_{i-k,t}^2-\mu_{X,\alpha})|} \ston{D} \sum_{i=1}^\infty \sum_{k=0}^m \eps_{\Gamma_i^{-2/\alpha} \theta_k \sum_j c_j^2 }.
\end{align*}
The extension to the case where $m=\infty$ follows analogously to the proof of \cite[Proposition 3.5 (case $2\leq\alpha<4$)]{Davis2011}. This establishes \cref{ma of partial sum 2} for $2\leq\alpha<4$, i.e.
\begin{align}\label{last eq}
\sum_{i=1}^p \eps_{a_{np}^{-2}|\sum_{k=0}^\infty \theta_{k}\sum_{t=1}^n (X_{i-k,t}^2-\mu_{X,\alpha})|} \ston{D} \sum_{i=1}^\infty \sum_{k=0}^\infty \eps_{\Gamma_i^{-2/\alpha} \theta_k \sum_j c_j^2 }.
\end{align}
Then one shows (i)-(iii) with $D$ replaced by $\tilde D$ by a straightforward combination of \eqref{last eq} and the approach used in the proof of \cref{theorem dependent rows} for $0<\alpha<2$.
\end{proof}

\section*{Acknowledgements}
The author thanks Richard Davis and Robert Stelzer for fruitful discussions on this topic. Their suggestions and comments improved this article considerably. The author further acknowledges the financial support of the Technische Universität München - Institute for Advanced Study, funded by the German Excellence Initiative, and the International Graduate School of Science and Engineering.


\begin{thebibliography}{12}
\providecommand{\natexlab}[1]{#1}
\providecommand{\url}[1]{\texttt{#1}}
\expandafter\ifx\csname urlstyle\endcsname\relax
  \providecommand{\doi}[1]{doi: #1}\else
  \providecommand{\doi}{doi: \begingroup \urlstyle{rm}\Url}\fi

\bibitem[Anderson and Zeitouni(2008)]{anderson2008}
G.~W. Anderson and O.~Zeitouni.
\newblock A law of large numbers for finite-range dependent random matrices.
\newblock \emph{Comm. Pure Appl. Math.}, 61\penalty0 (8):\penalty0 1118--1154,
  2008.
\newblock ISSN 0010-3640.
\newblock \doi{10.1002/cpa.20235}.
\newblock URL \url{http://dx.doi.org/10.1002/cpa.20235}.

\bibitem[Anderson et~al.(2010)Anderson, Guionnet, and Zeitouni]{Anderson2009}
G.~W. Anderson, A.~Guionnet, and O.~Zeitouni.
\newblock \emph{An introduction to random matrices}, volume 118 of
  \emph{Cambridge Studies in Advanced Mathematics}.
\newblock Cambridge University Press, Cambridge, 2010.
\newblock ISBN 978-0-521-19452-5.

\bibitem[Auffinger et~al.(2009)Auffinger, Ben~Arous, and
  P{\'e}ch{\'e}]{Auffinger2009}
A.~Auffinger, G.~Ben~Arous, and S.~P{\'e}ch{\'e}.
\newblock Poisson convergence for the largest eigenvalues of heavy tailed
  random matrices.
\newblock \emph{Ann. Inst. Henri Poincar\'e Probab. Stat.}, 45\penalty0
  (3):\penalty0 589--610, 2009.
\newblock ISSN 0246-0203.
\newblock \doi{10.1214/08-AIHP188}.
\newblock URL \url{http://dx.doi.org/10.1214/08-AIHP188}.

\bibitem[Bai and Silverstein(2010)]{Bai2010}
Z.~Bai and J.~W. Silverstein.
\newblock \emph{Spectral analysis of large dimensional random matrices}.
\newblock Springer Series in Statistics. Springer, New York, second edition,
  2010.
\newblock ISBN 978-1-4419-0660-1.
\newblock \doi{10.1007/978-1-4419-0661-8}.
\newblock URL \url{http://dx.doi.org/10.1007/978-1-4419-0661-8}.

\bibitem[Bhatia(1997)]{Bhatia1997}
R.~Bhatia.
\newblock \emph{Matrix analysis}, volume 169 of \emph{Graduate Texts in
  Mathematics}.
\newblock Springer-Verlag, New York, 1997.
\newblock ISBN 0-387-94846-5.

\bibitem[Billingsley(1999)]{Billingsley1999}
P.~Billingsley.
\newblock \emph{Convergence of probability measures}.
\newblock Wiley Series in Probability and Statistics: Probability and
  Statistics. John Wiley \& Sons Inc., New York, second edition, 1999.
\newblock ISBN 0-471-19745-9.
\newblock \doi{10.1002/9780470316962}.
\newblock URL \url{http://dx.doi.org/10.1002/9780470316962}.

\bibitem[Bose et~al.(2009)Bose, Hazra, and Saha]{bose2009}
A.~Bose, R.~S. Hazra, and K.~Saha.
\newblock Limiting spectral distribution of circulant type matrices with
  dependent inputs.
\newblock \emph{Electron. J. Probab.}, 14:\penalty0 2463--2491, 2009.
\newblock ISSN 1083-6489.

\bibitem[Davis et~al.(2011)Davis, Pfaffel, and Stelzer]{Davis2011}
R.~Davis, O.~Pfaffel, and R.~Stelzer.
\newblock Limit theory for the largest eigenvalues of sample covariance
  matrices with heavy-tails.
\newblock 2011.
\newblock URL \url{http://arxiv.org/abs/1108.5464}.

\bibitem[Hachem et~al.(2005)Hachem, Loubaton, and Najim]{hachem2005}
W.~Hachem, P.~Loubaton, and J.~Najim.
\newblock {The empirical eigenvalue distribution of a Gram matrix: from
  independence to stationarity}.
\newblock \emph{Markov Process. Related Fields}, 11\penalty0 (4):\penalty0
  629--648, 2005.
\newblock ISSN 1024-2953.

\bibitem[Resnick(2008)]{Resnick2008}
S.~I. Resnick.
\newblock \emph{Extreme values, regular variation and point processes}.
\newblock Springer Series in Operations Research and Financial Engineering.
  Springer, New York, 2008.
\newblock ISBN 978-0-387-75952-4.
\newblock Reprint of the 1987 original.

\bibitem[Soshnikov(2004)]{Soshnikov2004}
A.~Soshnikov.
\newblock Poisson statistics for the largest eigenvalues of {W}igner random
  matrices with heavy tails.
\newblock \emph{Electron. Comm. Probab.}, 9:\penalty0 82--91, 2004.
\newblock ISSN 1083-589X.
\newblock \doi{10.1214/ECP.v9-1112}.
\newblock URL \url{http://dx.doi.org/10.1214/ECP.v9-1112}.

\bibitem[Soshnikov(2006)]{Soshnikov2006}
A.~Soshnikov.
\newblock Poisson statistics for the largest eigenvalues in random matrix
  ensembles.
\newblock In \emph{Mathematical physics of quantum mechanics}, volume 690 of
  \emph{Lecture Notes in Phys.}, pages 351--364. Springer, Berlin, 2006.
\newblock \doi{10.1007/3-540-34273-7_26}.
\newblock URL \url{http://dx.doi.org/10.1007/3-540-34273-7_26}.

\end{thebibliography}

\end{document}